\documentclass[11pt]{article}
\usepackage[utf8]{inputenc}
\usepackage{multirow}
\usepackage{tikz}
\usetikzlibrary{positioning}
\usepackage[a4paper, margin=2.3cm]{geometry}
\usepackage{amsmath,amssymb,amsthm}
\usepackage{thmtools}
\usepackage{thm-restate}
\usepackage{xcolor}
\usepackage{parskip}
\usepackage{hyperref}
\hypersetup{hidelinks}
\usepackage{tikz-cd}
\usetikzlibrary{matrix,arrows.meta}
\usepackage{setspace}
\usepackage{makecell}
\usepackage{graphicx}
\usepackage{enumitem}
\usepackage{comment}
\usepackage{float}

\newtheorem{theorem}{Theorem}[section]
\newtheorem{remark}[theorem]{Remark}
\newtheorem{corollary}[theorem]{Corollary}

\newtheorem{lemma}[theorem]{Lemma}

\newtheorem{proposition}[theorem]{Proposition}

\newtheorem*{definition*}{Definition}

\newcommand{\F}{\mathbb F}

\newcommand{\abs}[1]{\left\lvert #1\right\rvert}

\title{On Erd\H{o}s--Falconer distance problem in even dimensions}
\author{Thang Pham \and Chun-Yen Shen\and Boqing Xue}
\date{}

\begin{document}

\maketitle

\begin{abstract}
Let $q$ be an odd prime power and $\mathbb{F}_q$ be the finite field of order $q$. We prove an extraction theorem for the Erdős--Falconer distance conjecture in even dimensions, showing that the conjecture for all even dimensions reduces to the planar case. As consequences, we obtain improved thresholds on the pinned distance problem and the distribution of triangles, achieving new records of $\frac{d}{2}+\frac{1}{4}$ over prime fields and $\frac{d+1}{2}+\frac{1}{10}$ over arbitrary finite fields, respectively.
\end{abstract}

\section{Introduction}
The distance set problem is a central topic in geometric measure theory,
harmonic analysis, and additive combinatorics.  Its discrete origins go
back to Erd\H{o}s's distinct distances problem
\cite{Erdos,GuthKatz}.  In the Euclidean setting, Falconer asked how large
the Hausdorff dimension of a compact set \(E\subset\mathbb R^d\) must be in
order to guarantee that
\[
    \Delta(E)=\{|x-y|:x,y\in E\}
\]
has positive Lebesgue measure.  His conjecture asserts that
\[
    \dim_{\mathrm H}(E)>\frac d2
    \qquad\Longrightarrow\qquad
    |\Delta(E)|>0.
\]
Since Falconer's foundational work \cite{Falconer}, the problem has driven
major developments in Fourier analysis and geometric measure theory,
including the work of Mattila \cite{Mattila}, Wolff \cite{Wolff},
Erdo\u{g}an \cite{Erdogan}, and more recent advances based on weighted
restriction and refined incidence geometry
\cite{DGOWWZ,GIOW,PeresSchlag,KS}. The finite field analogue, known as the Erd\H{o}s--Falconer distance
problem, asks for cardinality conditions on \(E\subset\F_q^d\)
guaranteeing that its distance set has size comparable to \(q\). Throughout, \(q\) denotes an odd prime power.

Let
\((V,Q)\) be a nondegenerate quadratic space over \(\F_q\), with
\(V=\F_q^{2m}\).  We write
\[
    B_Q(x,y)
    =
    \frac{Q(x+y)-Q(x)-Q(y)}{2}
\]
for the polar form of \(Q\).  For \(E\subset V\) and \(x\in E\), define
\[
    \Delta_Q(E)
    =
    \{Q(x-y):x,y\in E\},
    \qquad
    \Delta_{Q,x}(E)
    =
    \{Q(x-y):y\in E\}.
\]

The principal insight of this paper is that, for lower-bound problems
encoded by pairwise quadratic values, the essential geometric difficulty
in every even dimension already occurs in dimension two.  We make this
principle quantitative by extracting from every \(E\subset V\) a large
planar set whose complete matrix of quadratic distances is realized inside
\(E\).  Once the appropriate planar theorem is known, the passage to higher
even dimensions requires no new incidence geometry.

To formulate the extraction, set
\[
    H(u,v)=uv.
\]
For a nondegenerate quadratic form \(Q\) on \(\F_q^{2m}\), let \(P_Q\)
be a fixed representative on \(\F_q^2\) of the residual binary isometry
class determined by
\begin{equation}
\label{eq:residual-form}
    Q\simeq H^{\perp(m-1)}\perp P_Q.
\end{equation}
Here and throughout,
\[
    H^{\perp r}
    :=
    \underbrace{H\perp\cdots\perp H}_{r\text{ copies}}
\]
denotes the $r$-fold orthogonal sum of the hyperbolic binary form
$H(u,v)=uv$.
When \(m=1\), we use the convention \(H^{\perp0}=0\), so that
\(P_Q\simeq Q\).  For a general \(Q\), it is the planar theorem for this
particular residual form \(P_Q\) that governs the corresponding
even-dimensional problem.

Our main result is the following.

\begin{theorem} \label{maintheorem}
Let \((\F_q^{2m},Q)\) be a nondegenerate quadratic space of dimension
\(2m\), where \(q\) is an odd prime power and \(m\geq1\), and let \(P_Q\) be as in
\eqref{eq:residual-form}.  There is an absolute constant
\(C_0\geq1\) such that, for every
\(E\subset\F_q^{2m}\), there exist a set \(A\subset\F_q^2\) and an
injection \(\iota:A\to E\) satisfying
\begin{equation}
\label{eq:planar-extraction-size}
    |A|
    \geq
    \frac{1}{C_0}
    \frac{|E|/q^{m-1}}{1+|E|/q^{m+1}},
\end{equation}
and
\begin{equation}
\label{eq:planar-extraction-distance}
    P_Q(a-b)
    =
    Q\bigl(\iota(a)-\iota(b)\bigr)
    \qquad
    (a,b\in A).
\end{equation}
Consequently,
\[
    \Delta_{P_Q}(A)\subseteq\Delta_Q(E),
\]
and, for every \(a\in A\),
\[
    \Delta_{P_Q,a}(A)
    \subseteq
    \Delta_{Q,\iota(a)}(E).
\]
\end{theorem}

\paragraph{Planar universality.}
The simultaneous identity \eqref{eq:planar-extraction-distance} is the
essential feature of Theorem~\ref{maintheorem}: it preserves not only an
ordinary or pinned distance set, but every pairwise quadratic value on the
extracted set.  Let \(G\) be a fixed graph with vertex set
\(\{1,\ldots,r\}\) and edge set \(\mathcal E(G)\), and, for a subset \(E\)
of a quadratic space, define
\[
    \mathcal D_{G,Q}(E)
    :=
    \left\{
        \bigl(Q(x_i-x_j)\bigr)_{\{i,j\}\in\mathcal E(G)}
        :
        (x_1,\ldots,x_r)\in E^r
    \right\}.
\]
Theorem~\ref{maintheorem} gives, simultaneously for every fixed \(G\),
\[
    \mathcal D_{G,P_Q}(A)
    \subseteq
    \mathcal D_{G,Q}(E).
\]
Because \(\iota\) is injective, the same inclusion holds when the vertices
are required to be distinct; rooted and pinned variants are obtained by
distinguishing a vertex of \(G\).

At the level of cardinality exponents, a planar lower bound with threshold
\(q^\beta\), where \(\beta<2\), therefore transfers to dimension \(2m\)
with threshold
\[
    q^{m-1+\beta}
    =
    q^{d/2+\beta-1},
    \qquad d=2m,
\]
after adjusting the constants and for all sufficiently large \(q\).  This
is the precise sense in which it is enough to work in dimension two: all
nontrivial incidence-geometric input is planar, while the passage to higher
even dimensions is supplied by the same extraction theorem.  The reduction
applies to lower bounds encoded by pairwise quadratic values. However, it need not preserve affine structure, collinearity, orientation, or affine
independence.

\paragraph{The standard form.}
We now specialize to
\begin{equation}
\label{eq:fixed-form}
    V_m=\F_q^{2m},
    \qquad
    Q_m(x)=x_1^2+\cdots+x_{2m}^2.
\end{equation}
Put \(P_0(x,y)=x^2+y^2\).  The residual form of \(Q_m\) satisfies
\[
    P_{Q_m}
    \simeq
    \begin{cases}
        P_0,&\text{if \(-1\) is a square in \(\F_q\) or \(m\) is odd},\\
        H,&\text{if \(-1\) is nonsquare in \(\F_q\) and \(m\) is even}.
    \end{cases}
\]
Thus a planar theorem proved for both \(P_0\) and \(H\) yields a theorem
for \(Q_m\) in every even dimension, without a parity or congruence
restriction.  Notice that this two-form statement is specific to the
standard form \(Q_m\); for an arbitrary \(Q\), the required planar input is
the one for its actual residual form \(P_Q\).

For the standard form, the finite field Erd\H{o}s--Falconer problem asks
for cardinality conditions guaranteeing \(|\Delta_{Q_m}(E)|\gg q\), or,
in its pinned form, \(\Delta_{\mathrm{pin},Q_m}(E)\gg q\), where $\Delta_{\mathrm{pin},Q}(E)
 :=\max_{x\in E}|\Delta_{Q,x}(E)|$. The classical Fourier-analytic argument of Iosevich and Rudnev
\cite{IR} gives the sufficient exponent \(m+\frac12\), while the
conjectured threshold is \(m\).  In odd dimension $d$, by contrast, the
exponent \((d+1)/2\) is sharp in general: Hart et al.~\cite{hart}
constructed examples showing that no smaller uniform exponent can hold.  The strongest relevant progress has occurred in the plane.  The
unpinned exponent \(3/2\) was improved to \(4/3\) through Fourier
restriction, group actions, and incidence geometry~\cite{CEHIK,BHIPR}.
In the pinned setting, using a point-plane incidence bound due to Rudnev in \cite{Rud}, Murphy, Petridis, Pham, Rudnev, and
Stevens developed a
bisector-energy framework to prove the exponent \(5/4\) for \(P_0\) over prime
fields~\cite{MPPRS}. Therefore, our strongest consequence is obtained by proving the missing
pinned \(5/4\) theorem for the split plane \((\F_p^2,H)\) and combining it
with the known theorem for \(P_0\) through
Theorem~\ref{maintheorem}.

\begin{theorem}
\label{cor:pinned}
There are absolute constants \(C,c>0\) such that, for every odd prime \(p\),
every \(m\geq1\), and every \(E\subset\F_p^{2m}\),
\[
    |E|\geq C p^{m+1/4} \quad\Longrightarrow\quad \Delta_{\mathrm{pin},Q_m}(E)
    \geq c p.
\]
\end{theorem}

Corollary~\ref{cor:pinned} gives the pinned sufficient exponent
\[
\frac d2+\frac14
\]
for the standard quadratic form in every even dimension over every odd prime field.  The conclusion is pinned, and hence is stronger than the
corresponding unpinned assertion.  No higher-dimensional incidence estimate enters the deduction: the incidence-geometric input is entirely
contained in the two planar theorems.

The new planar ingredient is the split case.  Its main obstruction is the large zero level
\[
    H^{-1}(0)
    =
    \{u=0\}\cup\{v=0\}.
\]
A zero \(H\)-circle is the union of two isotropic lines and can contain a large proportion of the set.  We control these heavy fibres by a
directional-projection estimate, classify the exceptional projective lines arising in the split kinematic construction, and incorporate them into a modified bisector-energy argument.  This supplies precisely the missing planar input for Theorem~\ref{cor:pinned}.

Moreover, the planar unpinned \(4/3\) theorem is available for both \(P_0\) and \(H\) over arbitrary odd prime-power fields.  The extraction theorem therefore gives the following field-uniform consequence.

\begin{corollary}
\label{cor:unpinned}
There are absolute constants \(C',c'>0\) such that, for every odd prime power \(q\), every \(m\geq1\), and every \(E\subset\F_q^{2m}\),
\begin{equation}
\label{eq:main}
|E|\geq C'q^{m+1/3} \quad\Longrightarrow\quad |\Delta_{Q_m}(E)|\geq c'q.
\end{equation}
\end{corollary}

Thus, the planar exponent \(4/3\) becomes
\[
m-1+\frac43 = \frac d2+\frac13
\]
in dimension \(d=2m\).  Corollary~\ref{cor:unpinned} holds over every finite field of odd characteristic and for arbitrary sets, with no
product, regularity, or non-concentration hypothesis.  The pinned and unpinned results illustrate the modular nature of the framework: the
extraction step is unchanged, while the field of definition and the final exponent are inherited from the planar input.  The case \(m=1\) is the planar \(4/3\) theorem established in \cite{CEHIK}, while the present argument extends this exponent to every even dimension. 

More broadly, the transfer principle also applies to the quadratic edge-value data of fixed configurations.  We illustrate this with ordered triangles.  For a quadratic form \(Q\) and a subset \(E\) of its quadratic space, let
\[
    \mathcal T_Q(E)
    :=
    \left\{
        \bigl(
            Q(x_1-x_2),
            Q(x_1-x_3),
            Q(x_2-x_3)
        \bigr)
        :
        x_1,x_2,x_3\in E
    \right\}
\]
be the ordered triangle set determined by \(E\); equivalently, \(\mathcal T_Q(E)=\mathcal D_{K_3,Q}(E)\).  A planar theorem at exponent
\(8/5\), together with the required control of degenerate planar triangles, gives the following consequence.

\begin{theorem} \label{triangle}
There are absolute constants \(C'' , c'' >0\) such that, for every odd prime power \(q\), every \(m\geq1\), and every
\(E\subset\F_q^{2m}\),
\[
    |E|\geq C'' q^{m+3/5} \quad\Longrightarrow\quad 
    |\mathcal T_{Q_m}(E)|
    \geq
    c'' q^3.
\]
Consequently, \(E\) determines at least \(c'' q^3\) congruence classes of ordered triangles.
\end{theorem}

Theorem~\ref{triangle} is the first illustration of the transfer principle beyond two-point distance sets.  Here the planar exponent \(8/5\) becomes
\[
    m-1+\frac85
    =
    \frac d2+\frac35.
\]
More generally, the same mechanism applies to lower bounds for the quadratic edge-value patterns of any fixed graph, provided that the planar result is available for the residual binary forms that occur.  Complete graphs encode full edge-length data, trees encode prescribed edge-length patterns, and rooted stars encode pinned problems.  The automatic conclusion concerns the realized quadratic values; additional affine or nondegeneracy conditions require separate arguments.

Our triangle theorem should be compared with the earlier simplex results of \cite{BHIPR, CEHIK}.  The work \cite{CEHIK} gave the sufficient exponent $(d+k)/2$ for realizing every prescribed nondegenerate \(k\)-simplex; for triangles this is \(d/2+1\).  For determining a positive proportion of congruence classes, \cite{BHIPR} obtained the exponent $d-(d-1)/(k+1)$, which becomes \((2d+1)/3\) when \(k=2\), together with the sharper planar triangle exponent \(8/5\).

Using the extraction theorem, we lift this planar \(8/5\) result to every even dimension. Thus, for \(d=2m\), $|E|\gg q^{m+3/5}$ implies that \(E\) determines \(\gg q^3\) triangle congruence classes. Equivalently, we obtain the sufficient exponent $d/2+3/5$, which improves the previously known higher-dimensional bounds.

The extraction framework also differs from other methods previously used for even-dimensional distance problems.  We give two remarks here.

\begin{remark}
Corollary~\ref{cor:unpinned} recovers, at the level of the
exponent, the main result of Zhang~\cite{ZhangLP}. Zhang obtained the
threshold \(m+\frac{1}{3}\) through a linear-programming argument that
combines several properties of the Kloosterman matrix in an effective
way. While this method is particularly powerful for the unpinned
one-set distance problem, it appears less flexible for the purposes
considered here. In particular, it does not yield the exponent
\(m+\frac{1}{4}\). Moreover, as observed in~\cite{HT}, the argument
does not extend directly to two-set distance problems, pinned distance
problems, or more general multi-point configurations. Ham and Tran ~\cite{HT} subsequently developed the linear-programming framework
further in the two-set setting and formulated an \(L^1\)-averaged
Kloosterman conjecture whose resolution would imply the conjectured
distance threshold in even dimensions. By contrast, the extraction
principle developed here separates the dimension-reduction step from
the planar geometric input, allowing planar pinned and multi-point
results to be transferred without modifying the higher-dimensional
argument.
\end{remark}

\begin{remark}
Spherical restriction estimates constitute one of the principal
analytic tools in the finite field distance problem. Nevertheless,
even the full spherical restriction conjecture, if inserted into the
standard restriction-based argument, would yield only the sufficient
exponent
\[
    \frac{d}{2}+\frac{d}{2d+2};
\]
see~\cite{KPV}. Recently, the first and third authors~\cite{PX26}
improved upon the spherical Stein--Tomas restriction exponent in
dimension four over prime fields. As an application, they obtained
the sufficient exponent
\[
    \frac{5}{2}-\frac{1}{62}
\]
for the distance problem in \(\mathbb{F}_p^4\).
\end{remark}

\paragraph{Outline of our main ideas.}
The proof of Theorem~\ref{maintheorem} combines a low-collision isotropic
foliation with a two-dimensional Witt quotient.  We average over all
totally isotropic subspaces \(R\) of dimension \(m-1\).  A spectral
estimate for the zero quadric allows us to choose \(R\) so that
\[
    C_R
    :=
    \bigl|\{(x,y)\in E^2:x-y\in R\}\bigr|
    \ll
    |E|+\frac{|E|^2}{q^{m+1}}.
\]
This low-collision choice is the step that prevents too much of \(E\) from
collapsing when we pass to the quotient.

We next partition \(V\) into affine \(R^\perp\)-cosets and each such slice
into affine \(R\)-cosets.  The Cauchy--Schwarz inequality shows that one
\(R^\perp\)-slice meets \(E\) in at least
\[
    \gg
    \frac{|E|/q^{m-1}}{1+|E|/q^{m+1}}
\]
distinct \(R\)-cosets.  Choosing one point of \(E\) from each occupied
coset gives the injection in Theorem~\ref{maintheorem}.  The geometry of
the extracted set comes from
$
    R^\perp/R.
$
Because \(R\) is totally isotropic, this quotient is a nondegenerate
quadratic plane, and Witt cancellation identifies it with
\((\F_q^2,P_Q)\).  Furthermore, the quadratic value of a difference of
two points in the chosen \(R^\perp\)-slice depends only on their classes
modulo \(R\).  Hence
\[
    P_Q(a-b)
    =
    Q\bigl(\iota(a)-\iota(b)\bigr)
    \qquad(a,b\in A),
\]
which simultaneously transfers the quadratic data of every fixed
configuration on \(A\).

Conceptually, the proof isolates the only two tasks: a universal
even-dimensional extraction step and a genuinely planar geometric input.
The first is independent of the configuration under consideration; the
second contains all of the problem-specific incidence geometry.

The remainder of the paper is organized as follows.  In Section~2, we
collect the planar inputs and the necessary facts about quadratic spaces,
estimate zero-distance pairs, choose a low-collision totally isotropic
subspace, and identify the residual quadratic plane.  In Section~3, we
prove Theorems~\ref{maintheorem}, \ref{triangle} and deduce Corollary~\ref{cor:unpinned}. In Section~4, we establish the pinned
\(5/4\) theorem for the split form \(H\) and prove Theorem~\ref{cor:pinned}.

\section{Preliminaries}

The first result, for $P=P_0$, is taken from \cite{MPPRS}.

\begin{proposition}
\label{prop:prime-planar-both_1}
There are absolute constants $C_1,c_1>0$ such that, for every odd prime $p$, and every $A\subset\F_p^2$,
\[
|A|\ge C_1p^{5/4} \quad\Longrightarrow\quad \Delta_{\mathrm{pin},P_0}(A)\ge c_1 p.
\]
\end{proposition}

For $P=P_0$, the following proposition is
\cite[Theorem~1.6, equation~(1.3)]{BHIPR}.  The same argument for
the Minkowski form, given in the proof of
\cite[Corollary~1.8]{BHIPR}, yields the case $P=H$.
\begin{proposition}
\label{prop:planar}
There are absolute constants $C_2,c_2>0$ such that, for every odd
prime power $q$, every $P\in\{P_0,H\}$, and every
$A\subset\F_q^2$,
\[
|A|\ge C_2q^{4/3} \quad\Longrightarrow\quad  |\Delta_P(A)|\ge c_2q.
\]
\end{proposition}

The following proposition is \cite[Theorem 1.5]{BHIPR} for $P=P_0$, which also holds for $P=H$ by the comments in \cite[Section 6]{BHIPR}.

\begin{proposition}  \label{prop:planar-triangles}
There are absolute constants $C_3,c_3>0$ such that, for every odd prime power $q$, every $P\in\{P_0,H\}$, and every $A\subset\F_q^2$,
\[
|A|\geq C_3 q^{8/5} \quad\Longrightarrow\quad |\mathcal{T}_{\color{black} P}(A)|\geq c_3 q^3.
\]
\end{proposition}

Let $(V,Q)$ be a nondegenerate quadratic space of dimension $2m$,
and let $P_Q$ be as in \eqref{eq:residual-form}. 
Denote
\[
S_0:=\{v\in V:\, Q(v)=0\},\qquad  \nu_0(E):=|\{(x,y)\in E^2:Q(x-y)=0\}|.
\]
The following lemma is standard, a proof can be found in \cite[Lemma 4.1]{KohShen13}. Evaluating their formula at the zero frequency gives \eqref{eq:S0-size}, while applying Plancherel to the nonzero Fourier coefficients gives \eqref{eq:zero-mixing}.

\begin{lemma}
\label{lem:zero-spectrum}
There is a sign $\varepsilon\in\{1,-1\}$ such that
\begin{equation} \label{eq:S0-size}
|S_0|=q^{2m-1}+\varepsilon(q-1)q^{m-1}.
\end{equation}
And, for every $E\subset V$, 
\begin{equation}\label{eq:zero-mixing}
\nu_0(E)\ll \frac{|E|^2}{q}+q^m |E|.
\end{equation}
\end{lemma}

A subspace $R\le V$ is \emph{totally isotropic} if $Q(r)=0$ for every $r\in R$. Denote
\[
 C_R:=|\{(x,y)\in E^2:x-y\in R\}|.
\]
Two points have different in $R$ exactly when they lie in the same affine $R$-coset, so 
\[
C_R=\sum_{a+R\in V/R}|E\cap(a+R)|^2.
\]

\begin{lemma}
\label{lem:choose-R}
There is an absolute constant $C_0\ge1$ such that the
following holds. If $m\ge2$ and $E\subset V$, then there is a totally isotropic $(m-1)$-dimensional subspace $R\le V$ for which
\begin{equation} \label{eq:CR-bound}
C_R \le C_0\left(|E|+\frac{|E|^2}{q^{m+1}}\right).
\end{equation}
\end{lemma}

\begin{proof} 
Every nondegenerate quadratic form over a finite field of odd characteristic has anisotropic part of dimension at most two.  Hence, its Witt index in dimension $2m$ is $m$ or $m-1$, and totally isotropic
$(m-1)$-spaces exist.

Choose $R$ uniformly from the collection of all such subspaces.  By
Witt's extension theorem, the orthogonal group is transitive on the
nonzero isotropic vectors. Therefore, for every fixed nonzero vector $v$ with
$Q(v)=0$, the probability that $v\in R$ is a constant $\rho$. Double-counting incident pairs $(R,v)$ gives $\rho=\frac{q^{m-1}-1}{\abs{S_0}-1}$. 
In the split and nonsplit cases, 
\begin{align*}
|S_0|-1=(q^{m-1}+1)(q^m-1),\qquad \text{or}\qquad  |S_0|-1=(q^{m-1}-1)(q^m+1),
\end{align*}
respectively. It follows in either case that $\rho\le2q^{-m}$.

A nonisotropic vector belongs to no such $R$. The diagonal pairs always contribute, so
\[
\mathbb E_R C_R =|E|+\rho\bigl(\nu_0(E)-|E|\bigr) \le |E|+2q^{-m}\nu_0(E).
\]
Let $c_0$ be an absolute constant admissible in
\eqref{eq:zero-mixing}.  Then
\[
 \mathbb E_R C_R
 \le (1+2c_0)|E|+\frac{2c_0|E|^2}{q^{m+1}}
 \le (1+2c_0)
 \left(|E|+\frac{|E|^2}{q^{m+1}}\right).
\]
Thus at least one $R$ satisfies \eqref{eq:CR-bound}, with
$C_0=1+2c_0$.
\end{proof}

For a subspace $R\le V$, write
\[
 R^\perp=\{v\in V:B_Q(v,r)=0\text{ for every }r\in R\}.
\]

The following is the specialization of \cite[Lemma~8.10]{ElmanKarpenkoMerkurjev} to a totally isotropic subspace of codimension two in the Witt decomposition. The final identification follows from Witt cancellation \cite[Theorem~8.4]{ElmanKarpenkoMerkurjev}.

\begin{lemma}
\label{lem:residual-quotient}
Let $R\le V$ be totally isotropic with $\dim R=m-1$.  Then $R\subset R^\perp$, and $W:=R^\perp/R$ is a two-dimensional nondegenerate quadratic space under
\begin{equation} \label{eq:quotient-form}
\overline Q(u+R)=Q(u).
\end{equation}
Moreover, $(W,\overline Q)$ is isometric to
$(\F_q^2,P_Q)$.
\end{lemma}



\begin{lemma}
\label{lem:quotient-type}
We have 
\begin{equation} \label{eq:quotient-type}
P_{Q_m}\simeq
\begin{cases}
P_0,&\text{if $-1$ is a square in $\F_q$ or $m$ is odd},\\
H,&\text{if $-1$ is a nonsquare in $\F_q$ and $m$ is even}.
\end{cases}
\end{equation}
\end{lemma}

\begin{proof}
Over a finite field of odd characteristic, two quadratic spaces are isometric if and only if they have the same dimension and determinant square class. Since $\det Q_m=1$ and $\det H=-1/4\equiv-1\pmod{(\F_q^*)^2}$, taking determinants in the relation $Q_m\simeq H^{\perp(m-1)}\perp P_{Q_m}$ gives $\det P_{Q_m}\equiv(-1)^{m-1}$ $\pmod{(\F_q^*)^2}$. Since $\det P_0=1$ and $\det H\equiv-1$, the alternatives in \eqref{eq:quotient-type} follow.  When $-1$ is a square, $P_0$ and $H$ are isometric, and we use $P_0$ as the chosen representative.
\end{proof}

\section{Proof of Theorems \ref{maintheorem}, \ref{triangle}, and Corollary \ref{cor:unpinned}}

\begin{proof} [Proof of Theorem \ref{maintheorem}]
Without loss of generality, we assume that $E\neq \emptyset$ and $m\geq 2$. Otherwise the conclusion is immediate. 

Choose $R$ as in Lemma~\ref{lem:choose-R}.  By Lemma~\ref{lem:residual-quotient}, the quotient $W=R^\perp/R$ is a nondegenerate quadratic plane isometric to $(\F_q^2,P_Q)$. We fix an isometry $\psi:\, (\F_q^2,P_Q)\longrightarrow(W,\overline Q)$. There are $L:=|V/R^\perp|=q^{m-1}$ affine $R^\perp$-cosets. Partition each such coset further into affine $R$-cosets. In the $j$-th $R^\perp$-coset $(1\leq j\leq L)$, write
\[
n_{j,a}=|E\cap(a+R)|, \qquad n_j=\sum_a n_{j,a},
 \qquad c_j=\sum_a n_{j,a}^2, \qquad s_j:=|\{a+R:n_{j,a}>0\}|,
\]
where $a+R$ is an $R$-coset in that $R^\perp$-coset, and the sums are over all such cosets. One can verify that
\[
 \sum_jn_j=|E|,
 \qquad
 \sum_jc_j=C_R,
 \qquad
 n_j^2\le s_jc_j.
\]
Writing $S=\max_js_j$, one deduces by the Cauchy--Schwarz inequality that
\[
|E|^2\le L\sum_jn_j^2 \le LS\sum_jc_j =LSC_R.
\]
Using Lemma~\ref{lem:choose-R}, we obtain
\begin{equation}
 \label{eq:S-lower}
 S\ge \frac{|E|^2}{LC_R}\ge
 \frac1{C_0}
 \frac{|E|/q^{m-1}}{1+|E|/q^{m+1}}
\end{equation}
for some absolute constant $C_0$.

Now choose a coset $a_0+R^\perp$ attaining the maximum value $S$, and select an element $x_\mathfrak c\in E\cap \mathfrak c$ from each affine $R$-coset $\mathfrak c$ in $a_0+R^\perp$ that intersects $E$. Denote
\[
A_W:=\{(x_\mathfrak c-a_0)+R:\, \mathfrak c\cap E\neq \emptyset\},
\]
which is a subset of $W$. Noting that distinct affine $R$-cosets $\mathfrak c$ give distinct $(x_\mathfrak c-a_0)+R$ in $W$, so $|A_W|=S$. Take $A=\psi^{-1}(A_W)$. For $a\in A$, define $\iota(a)$ to be the selected point $x_\mathfrak c$ with $\psi(a)=(x_\mathfrak c-a_0)+R$. 
This defines an injection $\iota:A\to E$. Now equation
\eqref{eq:S-lower} proves \eqref{eq:planar-extraction-size}.

If $a,b\in A$, then
\[
 \psi(a-b)
 =\bigl(\iota(a)-\iota(b)\bigr)+R.
\]
Therefore, by \eqref{eq:quotient-form},
\[
 P_Q(a-b)
 =\overline Q\bigl(\psi(a-b)\bigr)
 =Q\bigl(\iota(a)-\iota(b)\bigr),
\]
which is \eqref{eq:planar-extraction-distance}.  Letting $(a,b)$ vary proves the unpinned inclusion.  Holding $a$ fixed proves the pinned inclusion for every $a\in A$.
\end{proof}

\begin{proof} [Proof of Corollary \ref{cor:unpinned}]
By Lemma~\ref{lem:quotient-type}, the quadratic $P_{Q_m}$ is isomorphic to either $P_0$ or $H$. We transport every extracted planar set through this isometry when applying Proposition~\ref{prop:planar}. The corresponding cardinalities are unchanged.

The function $t\longmapsto\frac{t}{1+t/q^{m+1}}$ $(t\geq 0)$ is increasing.  Hence, if $|E|\ge C'q^{m+1/3}$, Theorem \ref{maintheorem} gives a planar set $A$ satisfying
\begin{equation} \label{eq:cor1-A-size}
|A|\ge \frac1{C_0} \frac{C'q^{4/3}}{1+C'q^{-2/3}}.
\end{equation}
Choose $C'\ge4C_0C_2$. Since $C'$ is now fixed, there is an absolute quantity $q_0$ such that $C'q^{-2/3}\le1$ whenever $q\ge q_0$. For such $q$, \eqref{eq:cor1-A-size} implies $|A|\ge C_2q^{4/3}$. Proposition~\ref{prop:planar} and the distance inclusion in Theorem \ref{maintheorem} now give
\[
|\Delta_{Q_m}(E)| \ge|\Delta_{P_{Q_m}}(A)| \ge c_2q.
\]
Taking $c'=\min\{c_2,q_0^{-1}\}$ completes the proof uniformly in $q$ and $m$.
\end{proof}


\begin{proof}[Proof of Theorem \ref{triangle}]
 By
 Theorem~\ref{maintheorem}, there are $A\subset\F_q^2$ and an
 injection $\iota:A\to E$ such that
 \begin{equation} \label{eq:triangle-extracted-size}
|A| \ge \frac1{C_0} \frac{C'' q^{8/5}}{1+C''q^{-2/5}},
 \end{equation}
 and
 \begin{equation}
  \label{eq:triangle-extracted-isometry}
  P_{Q_m}(a-b)
  =Q_m\bigl(\iota(a)-\iota(b)\bigr)
  \qquad(a,b\in A).
 \end{equation}
Choose $C''\geq 4C_0C_3$, then there is a $q_0>1$ such that $|A|\geq C_3q^{8/5}$ whenever $q\geq q_0$. 
By Lemma~\ref{lem:quotient-type}, one has $P_{Q_m}$ is isomorphic to either $P_0$ or $H$. Therefore, one deduce by Theorem \ref{maintheorem} and Proposition \ref{prop:planar-triangles} that 
 \[ |\mathcal{T}_{P_{Q_m}}(A)|\geq c_3 q^3,\] 
Taking $c''=\min\{c_3,q_0^{-3}\}$, the theorem then follows. 
\end{proof}

\section{Proof of Theorem \ref{cor:pinned}}

To prove Theorem \ref{cor:pinned}, we need the following proposition for pinned distance with respect to $H$.

\begin{proposition}
\label{prop:prime-planar-both-2}
There are absolute constants $C_4,c_4>0$ such that, for every odd
prime $p$, and every $A\subset\F_p^2$,
\begin{equation} \label{eq:prime-planar-both}
|A|\ge C_4 p^{5/4} \quad\Longrightarrow\quad \Delta_{\mathrm{pin},H}(A)\ge c_4 p.
\end{equation}
\end{proposition}





Note that
\[
 H^{-1}(0)=\{u=0\}\cup\{v=0\}.
\]
We use the proof of Theorem~15 of \cite{MPPRS} for all nonzero
orbits and isolate the contribution of these two isotropic
directions.

Let $\mathfrak L$ be the set of all horizontal or vertical affine lines in $\F_p^2$. Denote
\[
 M_0(A)=\max_{\ell\in \mathfrak L}|A\cap\ell|.
\]

\begin{lemma}
\label{lem:split-fibre-pruning}
Let $A\subset\F_p^2$, $|A|\ge2p$, and $0<\delta<1$.  If $|\Delta_{H,a}(A)|<\delta p$ for every $a\in A$, then
\[
M_0(A)\le \frac{2\delta}{1-\delta}\,\frac{p^2}{|A|}.
\]
\end{lemma}

\begin{proof}
Choose a line $\ell_0\in \mathfrak L$ such that $|A\cap \ell_0|=M_0(A)$. Without loss of generality, we assume that $\ell_0$ is a vertical line given by $y=y_0$ for some $y_0\in \F_p$. Denote $A_0:=\{x\in \F_p:\, (x,y_0)\in A\}$, and put $B=A\setminus(A_0\times\{y_0\})$. Then $M_0(A)=|A_0|$. The map
\[
 (x,y)\longmapsto (y-y_0,x(y-y_0)):=(Y,Z)
\]
is injective on $B$. Denote its image by $\mathcal P$, satisfying $|\mathcal P|=|B|$. For $s\in A_0$,
\[
 H\bigl((x,y)-(s,y_0)\bigr)=(x-s)(y-y_0)=Z-sY.
\]
Hence
\[
|\{Z-sY:\, (Y,Z)\in \mathcal P\}|\leq |\Delta_{H,(s,y_0)}(A)|<\delta p
\]
Set $\mathcal E_s=|\{(Y,Z,Y',Z')\in\mathcal P^2:Z-sY=Z'-sY'\}|$. The Cauchy--Schwarz inequality and the above estimate give
$\mathcal E_s\ge |B|^2/(\delta p)$ for $s\in A_0$, and $\mathcal E_s\ge |B|^2/p$ for all $s\in \F_p$.

On summing $\mathcal E_s$ over $s\in\F_p$, diagonal pairs contribute $p|B|$. A distinct pair with $Y\ne Y'$ contributes for exactly one $s$,
and a distinct pair with $Y=Y'$ contributes for none.  Hence $\sum_{s\in\F_p}\mathcal E_s\le p|B|+|B|^2$. It follows that
\[
|A_0|\,\frac{(1-\delta)|B|^2}{\delta p} \leq \sum_{s\in\F_p}\left(\mathcal E_s-\frac{|B|^2}{p}\right)\le p|B|.
\]
Since $|A_0|\le p$ and $|A|\ge2p$, we have $|B|=|A|-|A_0|\ge |A|/2$. Now we conclude that 
\[
M_0(A)=|A_0| \leq \frac{\delta p^2}{(1-\delta)|B|} \leq \frac{2\delta}{1-\delta}\,\frac{p^2}{|A|}.
\]
\end{proof}

Define
\begin{equation}
 \label{eq:TH-def}
 \mathcal T_H(A)
 =
 |\{(a,b,c)\in A^3:
      H(a-b)=H(a-c),\ H(b-c)\ne0\}|.
\end{equation}

In the following, we bound $\mathcal T_H(A)$ from above. One can adapt the arguments developed in \cite{MPPRS} step by step for this task. However, in the following, we present an approach that the number $H$-isosceles triangles is at most the number of $P_0$-isosceles triangles with a controlled additional term depending on $M_0(A)$.

\begin{lemma}
\label{lem:split-isosceles}
There is an absolute constant $\widetilde{C}>0$ such that, if
$p\le |A|\le p^{4/3}$, then
\begin{equation}\label{eq:split-isosceles}
\mathcal T_H(A) \le \frac{|A|^3}{p}+\widetilde{C}\left(p^{2/3}|A|^{5/3}+p^{1/4}|A|^2 + p^{1/2}M_0(A)^{1/2}|A|^{3/2}\right).
\end{equation}
\end{lemma}

\begin{proof}
For $b,c\in A$ with $H(b-c)\ne0$, write
\[
\operatorname{Bis}_H(b,c)  :=\{x\in\F_p^2:\, H(x-b)=H(x-c)\}
\]
for their $H$-bisector.  This is a nonisotropic affine line and is the fixed line of the unique axial $H$-reflection interchanging
$b$ and $c$. For an affine line $\ell$, let $i_A(\ell)=|A\cap\ell|$, and let $b_A^*(\ell)$ be the number of ordered pairs $(b,c)\in A^2$ with $H(b-c)\ne0$ and $\operatorname{Bis}_H(b,c)=\ell$. We call $\ell$ \emph{supported} if $b_A^*(\ell)>0$. Then
\[
\mathcal T_H(A)=\sum_\ell i_A(\ell)b_A^*(\ell), \qquad \sum_\ell\left(i_A(\ell)-\frac{|A|}{p}\right)^2\le p|A|,
\]
and
\begin{equation}\label{eq:split-balanced-initial}
\mathcal T_H(A)-\frac{|A|^3}{p} \le \sum_\ell\left(i_A(\ell)-\frac{|A|}{p}\right)b_A^*(\ell).
\end{equation}

We use the rich-object decomposition from \cite[Sections~5.1--5.5]{MPPRS}.  Put
\[
k=\sqrt{8|A|}, \qquad  K=|A|^{4/3}p^{-2/3},
\]
and let $\Gamma$ consist of all $k$-rich affine lines and all $k$-rich nonzero $H$-circles, where a circle $\gamma$ is $k$-rich means that $|A\cap\gamma|\ge k$. Distinct members of $\Gamma$ meet in at most two points, so the private-point argument of \cite[Lemma~16]{MPPRS} gives
\begin{equation}  \label{eq:split-private}
\sum_{\gamma\in\Gamma}|A\cap\gamma|\le2|A|.
\end{equation}
For a centre $o$ and a direction $v$, let $A_o$ be the union of the sets $A\cap\gamma$ over the circle members of $\Gamma$ centred at $o$, and let $A_v$ be the analogous union over the line members of $\Gamma$ with direction $v$.  Set
\[
\mathcal C_1=\{o:\, |A_o|>K\}, \qquad \mathcal V_1=\{v:\, |A_v|>K\}.
\]
Equation~\eqref{eq:split-private} implies $|\mathcal C_1|+|\mathcal V_1|\ll\frac{|A|}{K}$. 

Let $\mathcal L_1$ consist of the supported bisectors which either pass through a point of $\mathcal C_1$ or are $H$-orthogonal to a
direction in $\mathcal V_1$, and let $\mathcal L_2$ contain the remaining supported bisectors.  Define
\[
\mathcal T_i'(A) := \sum_{\ell\in\mathcal L_i} \left(i_A(\ell)-\frac{|A|}{p}\right)b_A^*(\ell),
 \qquad i=1,2.
\]
Since $b_A^*(\ell)\le |A|$, we have
\begin{equation}\label{eq:split-rich-part}
|\mathcal T_1'(A)| \ll |A|^2\bigl(|\mathcal C_1|+|\mathcal V_1|\bigr) \ll\frac{|A|^3}{K}.
\end{equation}
Indeed, for a fixed $o$ the sum of $i_A(\ell)$ over the lines through $o$ is at most $|A|+p$, and
\[
\sum_{\ell\ni o} \left|i_A(\ell)-\frac{|A|}{p}\right|\ll |A|
\]
because $|A|\ge p$.  For a fixed direction, the class of affine lines $H$-orthogonal to that direction partitions the plane and satisfies the analogous bound $O(|A|)$.  Summing over $\mathcal C_1$ and $\mathcal V_1$ proves the display.

If
\[
\mathcal T_H(A)-\frac{|A|^3}{p} \ll |A|^2+\frac{|A|^3}{K},
\]
then \eqref{eq:split-isosceles} follows immediately. Otherwise, \eqref{eq:split-balanced-initial} and \eqref{eq:split-rich-part} allow us to put
\[
\mathcal T_2'(A)>0, \qquad \mathcal T_H(A)-\frac{|A|^3}{p}\ll \mathcal T_2'(A).
\]
Writing $\mathcal B_{2,H}^*(A) =\sum_{\ell\in\mathcal L_2}b_A^*(\ell)^2$, the line-variance estimate gives
\begin{equation}\label{eq:split-X-B2}
\mathcal T_2'(A)\le\sqrt{p|A|\,\mathcal B_{2,H}^*(A)}.
\end{equation}

We next prove the split analogue of the bisector-energy estimate in \cite[Sections~6.1--6.3]{MPPRS}.  For $r\in\F_p^*$, let $\mathcal S_r=\{(a,b)\in A^2:H(a-b)=r\}$. Every element of $\mathcal S_r$ has a unique representation $s=((x,y),(x+h,y+r/h))$ with $h\in\F_p^*$, and we associate to it the projective point $\mathfrak P_s=[1:h:x:hy]\in\mathbb P^3(\F_p)$. For $t=((X,Y),(X+g,Y+r/g))$, let
\[
 \Pi_t:\quad
 r\xi+g\zeta-gY\eta-rXw=0
\]
in projective coordinates $[w:\eta:\xi:\zeta]$. Both assignments are injective.  Indeed, $\mathfrak P_s$ recovers $h,x$ and then $y$, while the coefficient of $\xi$ in $\Pi_t$ is the fixed
nonzero value $r$, so the normalized coefficient vector recovers $g,X$ and then $Y$.

An orientation-reversing $H$-isometry has the form $\phi(u,v)=(\lambda v+a,\lambda^{-1}u+b)$. Solving $\phi(t)=s$ gives $\lambda=hg/r$.  Such an isometry is an
axial reflection precisely when $\phi^2$ is the identity, equivalently
$a+\lambda b=0$.  Substitution reduces this condition to $r(x-X)+hg(y-Y)=0$. Consequently,
\begin{equation}
 \label{eq:split-incidence}
 \mathfrak P_s\in\Pi_t
 \quad\Longleftrightarrow\quad
 s\text{ and }t\text{ are related by an axial }H\text{-reflection}.
\end{equation}

For a projective line $L$, define its point and plane multiplicities by
\[
 \mu_P(L)=|L\cap\{\mathfrak P_s:s\in \mathcal S_r\}|,
 \qquad
 \mu_\Pi(L)=|\{t\in \mathcal S_r:L\subset\Pi_t\}|.
\]
The incidences carried by $L$ are therefore at most $\mu_P(L)\mu_\Pi(L)$.  We place every projective line contained in $\{w=0\}$ in the exceptional family.  Such a line has $\mu_P(L)=0$, since every $\mathfrak P_s$ has $w=1$, and hence contributes no incidences.

It remains to classify a projective line $L$ not contained in $\{w=0\}$.  On its affine part put
\[
 h=\eta/w,\qquad x=\xi/w,\qquad z=\zeta/w;
\]
for a point $\mathfrak P_s$ one has $z=hy$.  If $h$ is nonconstant on $L$, use it as an affine parameter and write
\[
x=\alpha h+a_0,\qquad z=\beta h+b_0.
\]
With $o=(a_0,\beta)$, the first and second endpoints of $s$ lie, respectively, on
\begin{equation}
 \label{eq:split-two-levels}
 (u-a_0)(v-\beta)=\alpha b_0:=R_1,
 \qquad
 (u-a_0)(v-\beta)=(\alpha+1)(b_0+r):=R_2.
\end{equation}
If $h$ is constant, the two endpoints instead run on a pair of parallel affine lines.

The same classification controls the number of planes containing a projective line.  Indeed, the fixed duality
\[
 [c_w:c_\eta:c_\xi:c_\zeta]
 \longmapsto
 [c_\xi:-r c_\zeta:c_w:-r c_\eta]
\]
sends the coefficient vector $[-rX:-gY:r:g]$ of $\Pi_t$ to $[1:-g:-X:gY]$, which is the point associated with the negated segment $-t$. Negation carries the rich lines and circles for $A$ bijectively to
the corresponding rich objects for $-A$, preserving all multiplicities.  More precisely, $\mathcal C_1(-A)=-\mathcal C_1(A)$ and $\mathcal V_1(-A)=\mathcal V_1(A)$, where directions are identified projectively. Thus the same classification applies to plane multiplicity after this duality. Moreover, the nonconstant line above is contained in $\Pi_t$ exactly when
\begin{equation}
 \label{eq:split-dual-parameters}
 X=a_0+\frac{b_0}{r}g,
 \qquad
 Y=\beta+\frac{r\alpha}{g}.
\end{equation}
Thus the endpoints of $t$ satisfy the same two level equations \eqref{eq:split-two-levels}. When $h=h_0$ is constant on the affine part of $L$, write $x=x_0+u\tau$ and $z=z_0+v\tau$. If $uv=0$, no plane $\Pi_t$ contains this line.  If $uv\ne0$, then $g=-ru/v$, and the first endpoints of the corresponding segments $t$ lie on $vX-uh_0Y=vx_0-uz_0$. Their second endpoints lie on the corresponding parallel translate.

We now separate the possibilities in \eqref{eq:split-two-levels}.  If $R_1R_2\ne0$, we obtain the usual pair of concentric nonzero circles. Together with the constant-$h$ parallel-line case, these are exactly the ordinary families in the restricted-incidence proof of \cite{MPPRS}.  The ordinary exceptional family consists of the corresponding primal rich-annulus and rich parallel-line projective lines, together with their inverse-dual copies under the fixed duality above.  A line is included only when both endpoint circles, or both endpoint lines, belong to $\Gamma$; outside this union both its point and plane multiplicities are $O(k)$. In the nonconstant case every associated reflection axis has equation $u-\lambda v=a_0-\lambda\beta$ for some $\lambda\in\F_p^*$ and hence passes through $o$. In the parallel case the axis is $H$-orthogonal to the common direction. 

We now count the ordinary exceptional incidences directly, summed over $r\ne0$, with reflection axes in $\mathcal L_2$, and denote this total by $I_o$.  In the circle case, the number of choices of a pair $(a,o)$, where $a$ lies on a rich circle centred at $o$, is at most
\[
\sum_{\gamma\in\Gamma}|A\cap\gamma|\le2|A|.
\]
If the reflection axis belongs to $\mathcal L_2$, then $o\notin\mathcal C_1$, and hence the other endpoint has at most
$|A_o|\le K$ choices.  There are at most $p+1$ axes through $o$, and, after the segment and the axis have been chosen, its reflected segment is uniquely determined.  This gives $O(pK|A|)$ incidences. For parallel rich lines, one instead chooses a point together with the common direction $v$ in $O(|A|)$ ways, the other endpoint in $A_v$ in at most $K$ ways, and one of at most $p$ axes $H$-orthogonal to $v$.  This gives the same bound.  Repeating the same count after the above duality covers the inverse-dual copies. Consequently,
\begin{equation} \label{eq:split-ordinary-exceptional}
 I_o\ll pK|A|.
\end{equation}

There are three additional zero-level possibilities.

First, if one endpoint is fixed at $o$, the other endpoint runs on a nonzero circle $C_o(\rho)$.  These are fixed-endpoint stars. 
They occur precisely for
$(\alpha,b_0)=(0,0)$, when the first endpoint is fixed, and for
$(\alpha,b_0)=(-1,-r)$, when the second endpoint is fixed. 
Put
\[
 C_o(\rho)=\{x\in\F_p^2:H(x-o)=\rho\},
 \qquad \rho\in\F_p^*,
\]
and put $m(o,\rho)=|A\cap C_o(\rho)|$. A star based on a $k$-rich circle contributes $O(m(o,\rho)^2)$ incidences.  For a star arising from $\mathcal S_r$, its nonzero circle has radius $\rho=r$.  Hence each circle occurs in only an absolute number of orientations and dual copies over the entire
sum in $r$, and $m(o,\rho)\le p-1$. Let $I_{\mathrm{stars}}$ denote the total, summed over
$r\in\F_p^*$, of incidences $\mathfrak P_s\in\Pi_t$, with
$s,t\in\mathcal S_r$, which are carried by one of these fixed-endpoint
star projective lines based on a circle in $\Gamma$, or by one of their
inverse-dual copies under the fixed duality above, and for which the
corresponding reflection axis belongs to $\mathcal L_2$. Therefore
\begin{equation} \label{eq:split-stars}
I_{\mathrm{stars}} \ll\sum_{C_o(\rho)\in\Gamma}m(o,\rho)^2\ll p|A|
\end{equation}
by \eqref{eq:split-private}.  If the circle is not $k$-rich, both line multiplicities are $O(k)$.

Second, suppose exactly one of $R_1,R_2$ is zero, but the corresponding endpoint is not fixed.  The zero-level endpoint then
runs injectively on one horizontal or vertical component, while the other endpoint runs on a nonzero circle $C_o(\rho)$.  For fixed $r,o,\rho$ there are at most four such projective lines. Indeed, a zero first level comes from $\alpha=0$ or $b_0=0$, and a zero second level comes from $\alpha=-1$ or $b_0=-r$.  Once the zero component and the nonzero level $\rho$ are prescribed, the remaining parameter is determined. 
If $C_o(\rho)$ is not $k$-rich, both the point and plane multiplicities are $O(k)$.  If it is rich, include these lines in the exceptional family. Let $I_{\mathrm{mixed}}$ denote the analogous total, summed over
$r\in\F_p^*$, of incidences carried by these rich mixed zero/nonzero
projective lines or by their inverse-dual copies, again restricted to
incidences whose corresponding reflection axis belongs to
$\mathcal L_2$.  Each such line carries at most $m(o,\rho)^2$
incidences. Indeed, if the two segment parameters are $h,g$, then the associated reflection axis is
\[
 u-\lambda v=a_0-\lambda\beta,
 \qquad \lambda=hg/r,
\]
so it passes through $o=(a_0,\beta)$. Consequently, for an actual $\mathcal L_2$ incidence,
$o\notin\mathcal C_1$, so $|A_o|\le K$.  Since
\[
 \sum_{\rho:C_o(\rho)\in\Gamma}m(o,\rho)=|A_o|\le K,
\]
on summing over $r$ and using \eqref{eq:split-private}, we obtain
\begin{equation}
 \label{eq:split-mixed}
 I_{\mathrm{mixed}}
 \ll
 p\sum_{\substack{o\notin\mathcal C_1,\ \rho\in\F_p^*:\\
                   C_o(\rho)\in\Gamma}}m(o,\rho)^2
 \le
 pK\sum_{\gamma\in\Gamma}|A\cap\gamma|
 \ll pK|A|.
\end{equation}

Finally, both levels in \eqref{eq:split-two-levels} can vanish without a fixed endpoint.  Since $r\ne0$, this happens only for
\[
 (\alpha,b_0)=(0,-r)
 \qquad\text{or}\qquad
 (\alpha,b_0)=(-1,0).
\]
For $o=(o_1,o_2)$ and $r\ne0$, put
\begin{equation}
 \label{eq:split-eor}
 e_{o,r}
 :=
 \bigl|\{h\in\F_p^*:
 (o_1,o_2-r/h),(o_1+h,o_2)\in A\}\bigr|
\end{equation}
and 
\[
Z_H(A):=\sum_{o\in\F_p^2}\sum_{r\ne0}e_{o,r}^2.
\]
The two projective lines in question are
\[
 \{[w:h:o_1w:o_2h-rw]:[w:h]\in\mathbb P^1(\F_p)\},
 \qquad
 \{[w:h:o_1w-h:o_2h]:[w:h]\in\mathbb P^1(\F_p)\}.
\]
Equations~\eqref{eq:split-dual-parameters} show that each contains $e_{o,r}$ relevant points and is contained in $e_{o,r}$ relevant planes.  For example, a plane containing the first line corresponds to
\[
 ((o_1-g,o_2),(o_1,o_2+r/g));
\]
after reversing this segment and putting $h=-g$, one obtains exactly the segment counted by \eqref{eq:split-eor}.  The second line is identical, with the two roles reversed.  Their total contribution, including the dual copies, is therefore $O(Z_H(A))$.

For fixed $o,r$, the parameter $h$ maps injectively into both a horizontal and a vertical fibre of $A$, so
\[
 e_{o,r}\le M_0(A).
\]
Furthermore, every ordered pair $(a,b)\in A^2$ with $H(a-b)\ne0$ has a unique representation in \eqref{eq:split-eor}, namely
\[
 o=(a_1,b_2),\qquad h=b_1-a_1,\qquad r=H(b-a).
\]
It follows that
\begin{equation}
 \label{eq:split-Z-bound}
 Z_H(A)
 \le M_0(A)\sum_{o,r\ne0}e_{o,r}
 \le M_0(A)|A|^2.
\end{equation}

For each $r$, take the exceptional projective lines to be the ordinary rich annuli and parallel families, the rich stars, the rich mixed families, all the zero/zero lines above, their inverse-dual copies, and every projective line contained in $\{w=0\}$.  The last family contributes no incidences.  The preceding classification, applied also after the fixed duality, shows that outside this family one has
\[
 \mu_P(L)+\mu_\Pi(L)=O(k).
\]

For completeness, two distinct circles $C_a(r)$ and $C_{a'}(r)$ meet in at most two points.  If $d_r(b)=|\{a\in A:b\in C_a(r)\}|$, then
\[
\sum_{b\in A}d_r(b)^2  \le |\mathcal S_r|+2|A|^2.
\]
The Cauchy--Schwarz therefore gives
\[
 |\mathcal S_r|^2
 =\left(\sum_{b\in A}d_r(b)\right)^2
 \le |A|\bigl(|\mathcal S_r|+2|A|^2\bigr),
\]
and hence
\[
 |\mathcal S_r|\ll |A|^{3/2}\ll p^2
\]
because $|A|\le p^{4/3}$.  By the injectivity established above, the point and plane sets both have cardinality $|\mathcal S_r|$.  Thus the restricted point--plane incidence theorem \cite[Theorem~8]{MPPRS} applies with line multiplicity $O(k)$.

Let $I_{\ne0}$ denote, after summing over $r\ne0$, the number of incidences $\mathfrak P_s\in\Pi_t$ for which the unique axial $H$-reflection in \eqref{eq:split-incidence} has its axis in $\mathcal L_2$. Using \eqref{eq:split-ordinary-exceptional}--
\eqref{eq:split-Z-bound}, we obtain
\[
 I_{\ne0}
 \ll
 \sum_{r\ne0}\bigl(|\mathcal S_r|^{3/2}+k|\mathcal S_r|\bigr)
 +pK|A|+p|A|+M_0(A)|A|^2.
\]
Since $k\ll |A|^{1/2}$ and $\sum_{r\ne0}|\mathcal S_r|\le |A|^2$, this becomes
\begin{equation}
 \label{eq:split-nonzero-incidences}
 I_{\ne0}
 \ll
 \sum_{r\ne0}|\mathcal S_r|^{3/2}
 +|A|^{5/2}+pK|A|+M_0(A)|A|^2.
\end{equation}

It remains to relate this count to $\mathcal B_{2,H}^*(A)$ when a cross segment has $H$-length zero.  A contributing quadruple consists of two nonisotropic ordered bases $(a,c)$ and $(b,d)$ with the same bisector.  If $\sigma$ is the common reflection, then $\sigma(a)=c$, $\sigma(b)=d$, and hence
\[
 H(a-b)=H(c-d),
 \qquad
 H(a-d)=H(c-b).
\]
If the first common value is zero and the second is nonzero, the permuted segments $(a,d)$ and $(c,b)$ are counted injectively by $I_{\ne0}$, with the same axis in $\mathcal L_2$.  If both common values vanish, then, for fixed nonisotropic $(a,c)$,
\[
 b,d\in\{(a_1,c_2),(c_1,a_2)\}.
\]
Since $(b,d)$ is nonisotropic, there are at most two ordered choices.
Thus the residual all-zero contribution is $O(|A|^2)$, and
\begin{equation}
 \label{eq:split-B2-correction}
 \mathcal B_{2,H}^*(A)
 \ll
 \sum_{r\ne0}|\mathcal S_r|^{3/2}
 +|A|^{5/2}+pK|A|+M_0(A)|A|^2.
\end{equation}

We finish with the standard second-moment calculation.  If
\[
 H(a-b)=H(a-c)\ne0
 \quad\text{and}\quad
 H(b-c)=0,
\]
then $b=c$.  Indeed, if $b-c=(d,0)$ with $d\ne0$, subtraction forces
$a_2=c_2$ and hence makes the common value zero; the vertical case
is identical. 
For $a\in A$ and $r\ne0$, put $\nu_a(r):=|\{b\in A:H(a-b)=r\}|$. Then $|\mathcal S_r|=\sum_{a\in A}\nu_a(r)$, and the Cauchy--Schwarz gives
\[
 |\mathcal S_r|^2
 \le |A|\sum_{a\in A}\nu_a(r)^2.
\]
On summing in $r$, the rigidity just proved shows that the
off-diagonal triples on the right are counted by $\mathcal T_H(A)$;
the diagonal triples contribute at most $|A|^2$.  Consequently,
\[
 \sum_{r\ne0}|\mathcal S_r|^2
 \le |A|\bigl(\mathcal T_H(A)+|A|^2\bigr),
\]
and
\[
 \sum_{r\ne0}|\mathcal S_r|^{3/2}
 \le
 |A|\bigl(|A|(\mathcal T_H(A)+|A|^2)\bigr)^{1/2}.
\]
Using
$\mathcal T_H(A)-|A|^3/p\ll \mathcal{X}$ in
\eqref{eq:split-B2-correction} gives
\[
 \mathcal B_{2,H}^*(A)
 \ll
 |A|^{3/2}\mathcal{X}^{1/2}
 +\frac{|A|^3}{p^{1/2}}
 +|A|^{5/2}+pK|A|+M_0(A)|A|^2.
\]
Combining this with \eqref{eq:split-X-B2}, we obtain
\begin{align*}
 \mathcal{X}\ll{}&
 p^{1/2}|A|^{5/4}\mathcal{X}^{1/4}
 +p^{1/4}|A|^2
 +p^{1/2}|A|^{7/4}\\
 &+pK^{1/2}|A|
 +\sqrt{pM_0(A)}\,|A|^{3/2}.
\end{align*}
Young's inequality replaces the first term by
$O(p^{2/3}|A|^{5/3})$.  Since $|A|\le p^2$,
\[
 p^{1/2}|A|^{7/4}\le p^{2/3}|A|^{5/3},
\]
and our choice of $K$ gives
\[
 \frac{|A|^3}{K}=pK^{1/2}|A|=p^{2/3}|A|^{5/3}.
\]
Together with \eqref{eq:split-rich-part}, this proves
\eqref{eq:split-isosceles}.
\end{proof}

\begin{proof}[Proof of Proposition \ref{prop:prime-planar-both-2}]
Enlarge the eventual choice of $C_4$, if necessary, so that
$C_4\ge2$; hence $|A|\ge C_4p^{5/4}$ implies $|A|\ge2p$, as required
in Lemma~\ref{lem:split-fibre-pruning}.  It suffices, for all
sufficiently large $p$, to consider
\[
 C_4p^{5/4}\le |A|\le p^{4/3},
\]
because, once $C_4$ is fixed, a larger $A$ has a subset of size
$\lfloor p^{4/3}\rfloor\ge C_4p^{5/4}$.  A pin found in that subset
is also a pin of the original set, and its pinned distance set can
only increase.  We include the displayed size inequality in the
large-prime cutoff and treat the remaining primes at the end.
Fix a sufficiently small absolute
$0<\delta<1/16$ and suppose, for a contradiction, that every pin
determines fewer than $\delta p$ values.  Lemma~\ref{lem:split-fibre-pruning}
gives
\[
 M_0(A)\le
 \frac{2\delta}{1-\delta}\frac{p^2}{|A|}.
\]
Lemma~\ref{lem:split-isosceles} therefore yields
\begin{equation}
 \label{eq:TH-upper-bad}
 \mathcal T_H(A)
 \le\frac{|A|^3}{p}
 +\widetilde{C}\left(
 p^{2/3}|A|^{5/3}
 +p^{1/4}|A|^2
 +\sqrt{\frac{\delta}{1-\delta}}\,p^{3/2}|A|
 \right)
\end{equation}
with an absolute constant $\widetilde{C}$.

The zero cone is the union of a horizontal and a vertical line, so
$\nu_a(0)\le2p-1$.  If
$H(a-b)=H(a-c)\ne0$ and $H(b-c)=0$, then $b=c$:
for example, if $b-c=(d,0)$ with $d\ne0$, subtracting the two
equalities gives $d(a_2-c_2)=0$, which forces the common value to be
zero.  Thus
\[
 \mathcal T_H(A)+|A|^2
 \ge\sum_{a\in A}\sum_{t\ne0}\nu_a(t)^2.
\]
Applying the Cauchy-Schwarz inequality for each fixed \(a\in A\),
then summing over \(a\in A\), the bad-pin hypothesis and the preceding
estimate give
\begin{equation}
 \label{eq:TH-lower-bad}
 \mathcal T_H(A)
 \ge
 \frac{|A|(|A|-2p)^2}{\delta p}-|A|^2.
\end{equation}

At $|A|\ge C_4p^{5/4}$, the three error terms in \eqref{eq:TH-upper-bad}, divided by $|A|^3/p$, are respectively
\[
 O(C_4^{-4/3}),\qquad
 O(C_4^{-1}),\qquad
 O\!\left(\sqrt{\delta}\,C_4^{-2}\right).
\]
After fixing $0<\delta<1/16$, choose $C_4$ sufficiently large so
that the sum of the three normalized error
terms above is at most $1$.  Then
\eqref{eq:TH-upper-bad} is at most $2|A|^3/p$.  On the other hand, for
all sufficiently large $p$ the inequality
$|A|\ge C_4p^{5/4}$ gives $|A|-2p\ge |A|/2$, and
\eqref{eq:TH-lower-bad} yields
\[
 \frac{\mathcal T_H(A)}{|A|^3/p}
 \ge \frac1{4\delta}-\frac{p}{|A|}>2.
\]
This is the required contradiction. 
Let $p_0$ exceed both this cutoff and the subset-reduction
cutoff above.  For $p<p_0$, every nonempty pinned distance set contains
$0$ and hence has size at least $1\ge p/p_0$. Taking $C_4=\min\{\delta,p_0^{-1}\}$ proves the assertion for every odd prime.
\end{proof}


\begin{proof}[Proof of Theorem~\ref{cor:pinned}]
Again, Lemma~\ref{lem:quotient-type} shows that $P_{Q_m}$ is isomorphic to either $P_0$ or $H$.  We
transport extracted sets through this isometry whenever applying a
planar proposition.  Pulling a pin back through the isometry
preserves its pinned distance cardinality.

If $|E|\ge Cp^{m+1/4}$, the monotonicity used in the proof
of Corollary~\ref{cor:unpinned} and Theorem~\ref{maintheorem} give
\begin{equation}
 \label{eq:cor2-A-size}
 |A|\ge
 \frac1{C_0}
 \frac{Cp^{5/4}}
      {1+Cp^{-3/4}}.
\end{equation}
Choose
$C\ge4C_0C_4$.  There is an absolute
prime cutoff $p_0$ such that
$Cp^{-3/4}\le1$ for every $p\ge p_0$.  For such
primes, \eqref{eq:cor2-A-size} gives $|A|\ge C_4p^{5/4}$.
Propositions \ref{prop:prime-planar-both_1} and \ref{prop:prime-planar-both-2} supply $a\in A$ for
which $|\Delta_{P_{Q_m},a}(A)|\ge c_4p$. 
The pinned inclusion in Theorem~\ref{maintheorem}, with
$x=\omega(a)$, therefore yields $|\Delta_{Q_m, x}|\ge c_4p$. 
Taking $C=\min\{c_4,p_0^{-1}\}$ proves the theorem.
\end{proof}

\end{document}